\documentclass{article}

\usepackage{amsmath}
\usepackage{amssymb}
\usepackage{amsthm}

\newtheorem{theorem}{Theorem}

\newtheorem{proposition}{Proposition}
\newtheorem{corollary}{Corollary}

\title{Spectrum of multiplicative functions over powerful numbers}
\author{Tsz Ho Chan}
\date{}

\begin{document}
\maketitle

\begin{abstract}
Roughly speaking, the spectrum of multiplicative functions is the set of all possible mean values. In this paper, we are interested in the spectra of multiplicative functions supported over powerful numbers. We prove that its real logarithmic spectrum takes values from $- \sqrt{2} / (4 + \sqrt{2}) = -0.26160...$ to $1$ while it is known that the logarithmic spectrum of real multiplicative functions over all natural numbers takes values from $0$ to $1$. In the course of this study, we correct a proof of Granville and Soundararajan concerning contribution of small primes in the study of mean value of multiplicative functions. 
\end{abstract}

\section{Introduction and Main Results}

A function $f: \mathbb{N} \rightarrow \mathbb{C}$ is {\it multiplicative} if $f(a b) = f(a) f(b)$ for all relatively prime integers $a$ and $b$. A function $f: \mathbb{N} \rightarrow \mathbb{C}$ is {\it completely multiplicative} if $f(a b) = f(a) f(b)$ for all integers $a$ and $b$. Let $S$ be a subset of the unit disc $\mathbb{U} = \{ z \in \mathbb{C} : |z| \le 1 \}$. Let $\mathcal{F}(S)$ and $\hat{\mathcal{F}}(S)$ be the class of completely multiplicative functions and the class of multiplicative functions with $f(n) \in S$ for all $n \in \mathbb{N}$ respectively. Understanding the average behavior of multiplicative functions is one central theme in number theory. For example, Delange \cite{D}, Wirsing \cite{W} and Hal\'{a}sz \cite{H1} successively contributed to the study of average value of multiplicative functions.
\begin{theorem}[Delange, Wirsing \& Hal\'{a}sz] \label{thm-DWH}
Let $f \in \hat{\mathcal{F}}(\mathbb{U})$ be a multiplicative function. If there exists some real number $\tau$ such that the series
\begin{equation} \label{cond1}
\sum_{p} \frac{1 - \Re(f(p) p^{-i \tau})}{p}
\end{equation}
converges, then
\[
\frac{1}{x} \sum_{n \le x} f(n) = \frac{x^{i \tau}}{1 + i \tau} \prod_{p \le x} (1 - p^{-1}) \sum_{\nu = 0}^{\infty} f(p^{\nu}) p^{- \nu (1 + i \tau)} + o(1)
\]
as $x \rightarrow \infty$. If there is no real number $\tau$ for which the series \eqref{cond1} converges, then
\[
\frac{1}{x} \sum_{n \le x} f(n) = o(1).
\]
\end{theorem}
One way to prove the divergent case of Theorem \ref{thm-DWH} is through an effective Hal\'{a}sz' theorem \cite{H2} which was further refined by Montgomery \cite{M}. Recently, Granville, Harper and Soundararajan gave more intuitive and versatile arguments for Hal\'{a}sz' theorem in \cite{GHS1} and \cite{GHS2} together with some interesting applications.

\bigskip

People are also interested in the set of average values taken by (completely) multiplicative functions. Define
\[
\Gamma_N(S) := \Bigl\{ \frac{1}{N} \sum_{n \le N} f(n) : f \in \mathcal{F}(S) \Bigr\}, \; \; \hat{\Gamma}_N(S) := \Bigl\{ \frac{1}{N} \sum_{n \le N} f(n) : f \in \hat{\mathcal{F}}(S) \Bigr\}, 
\]
\[
\Gamma(S) = \lim_{N \rightarrow \infty} \Gamma_N(S), \; \;  \text{ and } \; \; \hat{\Gamma}(S) = \lim_{N \rightarrow \infty} \hat{\Gamma}_N(S)
\]
where $\lim_{N \rightarrow \infty} J_N = J$ for a sequence of subsets $J_N \subset \mathbb{U}$ means that $z \in J$ if and only if there is a sequence of points $z_N \in J_N$ with $z_N \rightarrow z$ as $N \rightarrow \infty$. Here $\Gamma(S)$ is called the {\it spectrum} of the set $S$. Granville and Soundararajan \cite{GS1} gave a general study on properties and geometric structures of the spectrum. One of their main results concerns about the spectrum of the interval $[-1, 1]$:
\begin{equation} \label{real-struct}
\Gamma([-1,1]) = [\delta_1, 1] \; \; \text{ where } \; \;
\delta_1 = 1 - 2 \log (1 + \sqrt{e}) + 4 \int_{1}^{\sqrt{e}} \frac{\log t}{t + 1} dt = -0.656999 \ldots .
\end{equation}
This means that
\[
\sum_{n \le x} f(n) \ge (\delta_1 + o(1)) x
\]
for all real-valued completely multiplicative functions with $-1 \le f(n) \le 1$. The proof of \eqref{real-struct} is partly based on the following beautiful structure theorem.
\begin{theorem}[Granville \& Soundararajan] \label{thm3}
For any closed subset $S$ of $\mathbb{U}$ with $1 \in S$, we have
\begin{equation} \label{struct}
\Gamma(S) = \Gamma_\Theta(S) \times \Lambda(S) \; \; \text{ and } \; \; \hat{\Gamma}(S) = \hat{\Gamma}_\Theta(S) \times \Lambda(S)
\end{equation}
where
\[
\Gamma_\Theta (S) := \lim_{x \rightarrow \infty} \Bigl\{ \prod_{p \le x} \Bigl(1 + \frac{f(p)}{p} + \frac{f(p^2)}{p^2} + \cdots \Bigr) \Bigl(1 - \frac{1}{p} \Bigr) : f \in \mathcal{F}(S) \Bigr\}
\]
and
\[
\hat{\Gamma}_\Theta (S) := \lim_{x \rightarrow \infty} \Bigl\{ \prod_{p \le x} \Bigl(1 + \frac{f(p)}{p} + \frac{f(p^2)}{p^2} + \cdots \Bigr) \Bigl(1 - \frac{1}{p} \Bigr) : f \in \hat{\mathcal{F}}(S) \Bigr\}
\]
are called the Euler product spectra, and $\Lambda(S)$ is the set of values of functions $\sigma : [0, \infty) \rightarrow \mathbb{U}$ that satisfy a certain integral equation (see equation \eqref{fcneq} or  \cite{GS1} for more details).
\end{theorem}

Granville and Soundararajan \cite{GS1} also investigated other notions of spectrum such as
\[
\lim_{x \rightarrow \infty} \Bigl\{ \sum_{n \le x} \frac{f(n)}{n^\sigma} \Big/ \sum_{n \le x} \frac{1}{n^\sigma} : f \in \mathcal{F}(S) ( \text{ or } \hat{\mathcal{F}}(S) ) \Bigr\}
\]
for any fixed $\sigma > 0$. These new notions are easier to study than $\Gamma(S)$ and $\hat{\Gamma}(S)$ with the most interesting case being $\sigma = 1$, the {\it logarithmic} spectrum:
\[
\Gamma_0(S) := \lim_{x \rightarrow \infty} \Bigl\{ \sum_{n \le x} \frac{f(n)}{n} \Big/ \sum_{n \le x} \frac{1}{n} : f \in \mathcal{F}(S) \Bigr\} \; \text{ and } \; \hat{\Gamma}_0(S) := \lim_{x \rightarrow \infty} \Bigl\{ \sum_{n \le x} \frac{f(n)}{n} \Big/ \sum_{n \le x} \frac{1}{n} : f \in \hat{\mathcal{F}}(S) \Bigr\}.
\]
By \cite[Theorem 8.4 \& Corollary 4]{GS1}, one has the following analogous structure result for logarithmic spectrum.
\begin{theorem}[Granville \& Soundararajan] \label{thm8.4}
For any closed subset $S$ of $\mathbb{U}$ with $1 \in S$,
\[
\Gamma_0 (S) = \Gamma_{\Theta}(S) \times \Lambda_0(S).
\]
Morover, when $S = [-1, 1]$,
\begin{equation} \label{log-spec0}
\Gamma_0([-1,1]) = [0,1].
\end{equation}
\end{theorem}
Subsequently, Granville and Soundararajan \cite{GS3} obtained more refined description of the lower endpoints for $\Gamma_0([-1, 1])$ and $\hat{\Gamma}_0([-1, 1])$: For large $x$,
\begin{equation} \label{log-spec1}
\sum_{n \le x} \frac{f(n)}{n} \ge - \frac{1}{(\log \log x)^{3/5}}
\end{equation}
for all
\[
f \in \mathcal{F}([-1,1]) \; \; \text{ or } \; \; f \in \hat{\mathcal{F}}([-1,1]) \text{ satisfying } \sum_{k = 1}^{\infty} \frac{1 + f(2^k)}{2^k} \gg (\log x)^{-1/20};
\]
and
\begin{equation} \label{log-spec2}
\sum_{n \le x} \frac{f(n)}{n} \ge \delta_1 \log 2 + o(1) = -0.4553... + o(1) \; \; \text{ for all } \; \; f \in \hat{\mathcal{F}}([-1,1])
\end{equation}
where the lower bound can be achieved if and only if
\[
\Bigl( \sum_{k = 1}^{\infty} \frac{1 + f(2^k)}{2^k} \Bigr) \log x + \sum_{3 \le p \le x^{1/(1 + \sqrt{e})}} \sum_{k = 1}^{\infty} \frac{1 - f(p^k)}{p^k} + \sum_{x^{1/(1 + \sqrt{e})} \le p \le x} \frac{1 + f(p)}{p} = o(1).
\]
In particular, the obvious set inclusion $\Gamma_0([-1, 1]) \subset \hat{\Gamma}_0([-1, 1])$ together with \eqref{log-spec0}, \eqref{log-spec1} and \eqref{log-spec2} imply
\[
\hat{\Gamma}_0([-1,1]) = [0,1]
\]
as $\sum_{n \le x} 1/n = \log x + O(1)$. It is also worth mentioning that inequality \eqref{log-spec1} has recently been improved by Kerr and Klurman \cite{KK} to
\begin{equation} \label{eq-KK}
\sum_{n \le x} \frac{f(n)}{n} \ge - \frac{1}{(\log \log x)^{1 - \epsilon}}
\end{equation}
for $f \in \mathcal{F}([-1,1])$ and any $\epsilon  > 0$.

\bigskip

In this paper, we are interested in the spectrum of multiplicative functions supported over powerful numbers. A number $n$ is {\it powerful} or {\it squarefull} if its prime factorization $n = p_1^{a_1} p_2^{a_2} \cdots p_r^{a_r}$ satisfies $a_i \ge 2$ for all $1 \le i \le r$. Similarly, a number $n$ is $k$-{\it full} if the exponents $a_i \ge k$ for all $1 \le i \le r$. For example, $72 = 2^3 \cdot 3^2$ is squarefull and $243 = 3^5$ is $5$-full. Let $\mathcal{Q}_k$ be the set of all positive $k$-full numbers and $Q_k(x)$ be its counting function. It is well known that
\begin{equation} \label{asymptotic}
Q_k(x) := \mathop{\sum_{n \le x}}_{n \in \mathcal{Q}_k} 1 = \prod_{p} \Bigl(1 + \sum_{m = k+1}^{2k - 1} \frac{1}{p^{m/k}} \Bigr) x^{1/k} + O(x^{1/(k+1)})
\end{equation}
where the product is over all primes (see \cite{ES} or \cite{BG} for example).

\bigskip

Let $\hat{\mathcal{F}}_k(S)$ denote the class of multiplicative functions with $f(n) \in S \subset \mathbb{U}$ that are supported over $k$-full numbers (i.e., $f(n) = 0$ for $n \not\in \mathcal{Q}_k$). Define the regular spectrum and {\it ``logarithmic"} spectrum of multiplicative functions over powerful numbers by
\[
\hat{\Gamma}_{\mathcal{Q}_2}(S) := \lim_{x \rightarrow \infty} \Bigl\{ \sum_{n \le x} f(n) \; \; \Big/ \sum_{n \le x, n \in \mathcal{Q}_2} 1 : f \in \hat{\mathcal{F}}_2(S) \Bigr\}
\]
and
\[
\hat{\Gamma}_{0, \mathcal{Q}_2}(S) := \lim_{x \rightarrow \infty} \Bigl\{ \sum_{n \le x} \frac{f(n)}{n^{1/2}} \; \; \Big/ \sum_{n \le x, n \in \mathcal{Q}_2} \frac{1}{n^{1/2}} : f \in \hat{\mathcal{F}}_2(S) \Bigr\}
\]
respectively. The reason for the exponent $1/2$ is that $Q_2(x) = \prod_p (1 + \frac{1}{p^{3/2}})  \sqrt{x} + O(x^{1/3})$ implies
\begin{equation} \label{log-powerfulcount}
\sum_{n \le x, n \in \mathcal{Q}_2} \frac{1}{n^{1/2}} = \prod_{p} \Bigl(1 + \frac{1}{p^{3/2}} \Bigr) \log \sqrt{x} + O(1)
\end{equation}
via partial summation. Also, we do not consider completely multiplicative functions supported over powerful numbers because there is only one such function, namely $f(1) = 1$ and $f(n) = 0$ for all $n > 1$, as $f(p) = 0$ for all prime $p$. Now, we are ready to state our main results.
\begin{theorem} \label{thm-logstruct}
For any closed subset $S$ of $\mathbb{U}$ with $1 \in S$,
\begin{equation*}
\hat{\Gamma}_{\mathcal{Q}_2} (S) = \hat{\Gamma}_{\Theta} (S) \times \Lambda (S) \; \; \text{ and } \; \; \hat{\Gamma}_{0, \mathcal{Q}_2} (S) = \hat{\Gamma}_{\Theta} (S) \times \Lambda_0 (S) 
\end{equation*}
where
\begin{align*}
\hat{\Gamma}_{\Theta} (S) := \lim_{x \rightarrow \infty} \biggl\{ & \prod_{p \le x} \Bigl(1 + \frac{f(p^3)}{p^{3/2}} + \frac{f(p^4) - f(p^2) f(p^2)}{p^{4/2}} + \frac{f(p^5) - f(p^3) f(p^2)}{p^{5/2}} + \cdots \Bigr) \Bigl( \frac{1 - \frac{1}{p}}{1 - \frac{f(p^2)}{p}} \Bigr) \Bigl( \frac{1}{1 + \frac{1}{p^{3/2}}} \Bigr) \\
& :  f \in \hat{\mathcal{F}}_2 (S) \biggr\} 
\end{align*}
is some kind of modified Euler product spectrum.
\end{theorem}
\begin{theorem} \label{thm-main}
When $S = [-1, 1]$, we have
\[
\hat{\Gamma}_{\mathcal{Q}_2} ([-1, 1]) = [\delta_1, 1] \; \; \text{ and } \; \; \hat{\Gamma}_{0, \mathcal{Q}_2}([-1, 1]) = [\delta_2, 1] = [-0.26120..., 1]
\]
where $\delta_2 = - \frac{\sqrt{2}}{4 + \sqrt{2}}$.
\end{theorem}
The conditions on \eqref{log-spec1} and \eqref{log-spec2} as well as the proof of Theorem \ref{thm-main} indicate that logarithmic spectra are greatly influenced by the values of the multiplicative functions on powers of $2$. Hence, we have the following result for multiplicative functions supported over odd powerful numbers which is similar to \eqref{eq-KK} and follows easily from applying the argument of Kerr and Klurman \cite{KK} to \eqref{temp1} and the observation \eqref{ip}.
\begin{corollary} \label{thm-odd}
Suppose $f \in \hat{\mathcal{F}}([-1, 1])$ with $f(2^k) = 0$ for all integer $k \ge 2$. For any $\epsilon > 0$,
\[
\sum_{n \le x} \frac{f(n)}{\sqrt{n}} \ge - \frac{1}{(\log \log x)^{1 - \epsilon}}
\]
for large enough $x$.
\end{corollary}

One way to interpret Theorem \ref{thm-main} is as follows. The set of powerful numbers $\mathcal{Q}_2$ is generated by the prime powers
\[
2^2, \; 3^2, \; 5^2 \; \ldots; \; 2^3, \; 3^3, \; 5^3, \; \ldots; \; 2^4, \; 3^4, \; 5^4, \; \ldots; \; \ldots 
\]
multiplicatively. In the ``logarithm" spectrum, we take square root of $n$. Hence, we may think of multiplicative functions supported over powerful numbers as functions over the following new multiplicative building blocks
\begin{equation} \label{block}
2, \; 3, \; 5 \; \ldots; \; 2^{3/2}, \; 3^{3/2}, \; 5^{3/2}, \; \ldots; \; 2^2, \; 3^2, \; 5^2, \; \ldots; \; \ldots 
\end{equation}
in contrast to the usual prime powers building blocks
\[
2, \; 3, \; 5 \; \ldots; \; 2^{2}, \; 3^{2}, \; 5^{2}, \; \ldots; \; 2^3, \; 3^3, \; 5^3, \; \ldots; \; \ldots 
\]
for ordinary multiplicative functions. Thus, Theorem \ref{thm-main} roughly says that, by perturbing the set of prime powers to \eqref{block}, the behavior of its logarithmic spectrum changes from $[0, 1]$ to $[-0.26120..., 1]$ while its regular spectrum stays unchanged.

\bigskip

The paper is organized as follows. First, we collect some detailed structure results on various spectra as well as a couple of slow variation results on multiplicative functions. Then, we give a corrected proof to a result of Granville and Soundararajan in \cite{GS1} on separation of contribution of small primes from mean values of multiplicative functions. This proof is related to later constructions on multiplicative functions supported over powerful numbers. At the end, we use these and ideas from \cite{GS1} and \cite{GS3} to prove Theorems \ref{thm-logstruct} and \ref{thm-main}.

\bigskip

{\bf Notation.} Throughout the paper, $p$ stands for a prime number. The symbols $f(x) = O(g(x))$, $f(x) \ll g(x)$ and $g(x) \gg f(x)$ are equivalent to $|f(x)| \leq C g(x)$ for some constant $C > 0$. The symbol $f(x) = o(g(x))$ means that $\lim_{x \rightarrow \infty} f(x)/g(x) = 0$.

\section{More detailed structure and slow variation results}

First, we recall some structure results for the Euler product spectrum $\Gamma_\Theta(S)$. Define 
\[
\mathcal{E}(S) := \left\{ e^{-k (1 - \alpha)} : \; k \ge 0, \; \alpha \text{ is in the convex hull of } S \right\}.
\]
\begin{theorem}[Granville \& Soundararajan] \label{thm4}
For all closed subsets $S$ of $\mathbb{U}$ with $1 \in S$,
\[
\mathcal{E}(S) \subset \Gamma_{\Theta}(S) \times \mathcal{E}(S) = \Gamma_{\Theta}(S) \subset \mathcal{E}(S) \times [0, 1].
\]
Moreove, if the convex hull of $S$ contains a real point other than $1$, one has
\[
\Gamma_{\Theta}(S) = \mathcal{E}(S) = \mathcal{E}(S) \times [0, 1].
\]
\end{theorem}
This comes from \cite[Theorem 4]{GS1} and implies that $\Gamma_\Theta(S)$ is starlike. In particular when $S = [-1, 1]$, we have $\mathcal{E}([-1, 1]) = [0, 1]$ and
\begin{equation} \label{EulerSpec-use}
\Gamma_{\Theta}([-1, 1]) = [0, 1].
\end{equation}

We also need some structure results on the logarithmic spectrum $\Gamma_0(S)$. For a closed subset $S$ of $\mathbb{U}$ with $1 \in S$, denote $S^*$ to be the convex hull of $S$. Let $K(S)$ denote the class of measurable functions $\chi : [0, \infty) \rightarrow S^*$ with $\chi(t) = 1$ for $0 \le t \le 1$. By \cite[Theorem 3.3]{GS1}, associated to each $\chi$ there is a unique function $\sigma : [0, \infty) \rightarrow \mathbb{U}$ satisfying the following integral equation:
\begin{equation} \label{fcneq}
u \sigma(u) = \int_{0}^{u} \sigma(u - t) \chi(t) dt
\end{equation}
for $u > 1$ and with the initial condition $\sigma(u) = 1$ for $0 \le u \le 1$.

\bigskip

Define $\Lambda_0(S)$ to be the set of values
\[
\Bigl\{ \frac{1}{u} \int_{0}^{u} \sigma(t) dt : u > 0 \text{ and } \sigma \text{ satisfies } \eqref{fcneq} \text{ for some } \chi \in K(S) \Bigr\}
\]
Define $\mathcal{R}$ to be the closure of the convex hull of the points $\prod_{i = 1}^{n} \frac{1 + s_i}{2}$, for all $n \ge 1$, and all choices of points $s_1, \ldots, s_n$ lying in $S^*$. In addition to Theorem \ref{thm8.4}, one has the following.
\begin{theorem}[Granville \& Soundararajan] \label{thm8}
For any closed subset $S$ of $\mathbb{U}$ with $1 \in S$,
\[
\Lambda_0(S) = \Lambda_0(S) \times \mathcal{E}(S), \; \; \Lambda_0(S) \subset \Gamma_0(S) \subset \Lambda_0(S) \times [0,1], \; \; \text{
and } \; \; \Lambda_0(S) \subset \mathcal{R}.
\]
\end{theorem}
This comes from the proof of Theorem 8 in \cite{GS1}. One can check that $1 \in \Lambda_0(S)$. Hence, when $S = [-1, 1]$, we have $\mathcal{R} = [0, 1]$ and 
\[
[0, 1] = \mathcal{E}([-1, 1]) \subset \Lambda_0([-1, 1]) \times \mathcal{E}([-1, 1]) =  \Lambda_0([-1, 1]) \subset \mathcal{R} = [0, 1]. 
\]
In particular,
\begin{equation} \label{LambdaSpec-use}
\Lambda_0([-1, 1]) = [0, 1].
\end{equation}

Finally, we recall the following slow variation results on multiplicative functions.
\begin{proposition} \label{prop4.1}
Let $f$ be a multiplicative function with $|f(n)| \le 1$ for all $n$. Let $x$ be large, and suppose $1 \le y \le x$. Then
\[
\Big| \frac{1}{x} \sum_{n \le x} f(n) - \frac{1}{x/y} \sum_{n \le x/y} f(n) \Big| \ll \frac{\log 2y}{\log x} \exp \Bigl(\sum_{p \le x} \frac{|1 - f(p)|}{p} \Bigr).
\]
\end{proposition}
\begin{proposition} \label{slowdecay}
Let $f$ be a multiplicative function with $|f(n)| \le 1$ for all $n$. Then, for $1 \le w \le x/10$, we have
\[
\frac{1}{x} \Big| \sum_{n \le x} f(n) \Big| - \frac{w}{x} \Big| \sum_{n \le x/w} f(n) \Big| \ll \Bigl( \frac{\log 2w}{\log x} \Bigr)^{1 - 2/\pi} \log \Bigl( \frac{ \log x}{\log 2w} \Bigr) + \frac{\log \log x}{(\log x)^{2 - \sqrt{3}}}.
\]
\end{proposition}
The above propositions come from Proposition 4.1 in \cite{GS1} and Corollary 3 in \cite{GS2} respectively.

\section{A correction to Granville and Soundararajan's proof}

To obtain spectrum results on multiplicative functions, Granville and Soundararajan \cite{GS1} derived the following proposition which was used to separate the contribution of small primes.
\begin{proposition} \label{prop4.5}
For any multiplicative function $f$ with $|f(p^k)| \le 1$ for every prime power $p^k$, let
\[
s(f, x) := \sum_{p \le x} \frac{|1 - f(p)|}{p}.
\]
For any $1 > \epsilon \ge \log 2 / \log x$, let $g$ be the completely multiplicative function with
\begin{equation} \label{g}
g(p) := \left\{ \begin{array}{ll} 1, & \text{ if } p \le x^\epsilon; \\ f(p),  & \text{ if } p > x^\epsilon. \end{array} \right.
\end{equation}
Then
\[
\frac{1}{x} \sum_{n \le x} f(n) = \Theta(f, x^\epsilon) \frac{1}{x} \sum_{m \le x} g(m) + O \bigl( \epsilon \exp( s(f,x)) \bigr),
\]
where the implicit constant is absolute.
\end{proposition}
In \cite{GS1}, Granville and Soundararajan defined a multiplicative function $h$ by $h(p^k) = f(p^k) - f(p^{k-1})$ if $p \le x^\epsilon$, and $h(p^k) = 0$ otherwise. They claimed that $f(n) = \sum_{m | n} h(n/m) g(m)$. However, this is only true when $f$ is completely multiplicative but is false in general. For example, say
\[
f(p) = 1, \; \; f(p^2) = -1 \text{ for some prime } p > x^\epsilon. 
\]
Then,
\[
\sum_{m | p^2} h\Bigl( \frac{p^2}{m} \Bigr) g(m) = h(p^2) g(1) + h(p) g(p) + h(1) g(p^2) = 0 \cdot g(1) + 0 \cdot g(p) + 1 \cdot f(p)^2 = 1 \neq - 1 = f(p^2).
\]
Thus, one needs to be more careful with its proof. It is worth mentioning that most of the results in \cite{GS1} dealt with completely multiplicative functions and were not affected by the above fault. However, our modified proof below is needed for the structure result $\hat{\Gamma}(S) = \hat{\Gamma}_\Theta(S) \times \Lambda(S)$ in \eqref{struct}. Also, subsequent work like \cite[Theorem 2]{GS2} required Proposition \ref{prop4.5} for all multiplicative functions. Therefore, we shall give a completely accurate proof for Proposition \ref{prop4.5}. 

\begin{proof}[Proof of Proposition \ref{prop4.5}]
We define a multiplicative function $\tilde{h}$ and a completely multiplicative function $\tilde{g}$ by
\[
\tilde{h}(p^k) := f(p^k) - f(p^{k-1}) f(p), \; \text{ and } \; \; \tilde{f}(p^k) = f(p)^k
\]
respectively for all prime $p$ and positive integer $k$. Then, one can check that $f(n) = \sum_{m | n} \tilde{h}(n/m) \tilde{f}(m)$ as
\begin{align*}
\sum_{m | p^k} \tilde{h}\Bigl( \frac{p^k}{m} \Bigr) \tilde{f}(m) &= \tilde{h}(p^k) \tilde{f}(1) + \tilde{h}(p^{k-1}) \tilde{f}(p) + \cdots + \tilde{h}(p) \tilde{f}(p^{k-1}) + \tilde{h}(1) \tilde{f}(p^k) \\
&= \sum_{i = 0}^{k - 1} ( f(p^{k - i}) - f(p^{k- i - 1}) f(p)) f(p)^i + f(p)^k \\
&= \sum_{i = 0}^{k-1} f(p^{k - i}) f(p)^i - \sum_{j = 1}^{k}f(p^{k- j}) f(p)^{j} + f(p)^k = f(p^k) \\
\end{align*}
Hence, we have
\[
\frac{1}{x} \sum_{n \le x} f(n) = \sum_{n \le x} \frac{\tilde{h}(n)}{n} \cdot \frac{n}{x} \sum_{m \le x/n} \tilde{f}(m).
\]
Now, with the completely multiplicative function $g$ defined by \eqref{g} and the new multiplicative function
\[
h(p^k) = \left\{ \begin{array}{ll} f(p)^k - f(p)^{k-1}, & \text{ if } p \le x^\epsilon; \\ 0,  & \text{ otherwise}, \end{array} \right.
\]
one can check that $\tilde{f}(n) = \sum_{m | n} h(n/m) g(m)$ as
\[
\sum_{m | p^k} h \Bigl( \frac{p^k}{m} \Bigr) g(m) = \left\{ \begin{array}{ll} (f(p)^k - f(p)^{k-1}) + (f(p)^{k-1} - f(p)^{k-2}) + \cdots + (f(p) - f(1)) + 1, & \text{ if } p \le x^\epsilon; \\
0 + 0 + \cdots + 0 + f(p)^k, & \text{otherwise} \end{array} \right.
\]
which equals to $f(p)^k = \tilde{f}(p^k)$ in all cases. Therefore, by Proposition \ref{prop4.1},
\begin{align} \label{doubleconvolute}
\frac{1}{x} \sum_{n \le x} f(n) &= \sum_{n \le x} \frac{\tilde{h}(n)}{n} \sum_{m \le x/n} \frac{h(m)}{m} \cdot \frac{1}{x / (mn)} \sum_{k \le x/(m n)} g(k) \nonumber \\
&= \sum_{n \le x} \frac{\tilde{h}(n)}{n} \sum_{m \le x/n} \frac{h(m)}{m} \cdot \Bigl[ \frac{1}{x} \sum_{k \le x} g(k) + O \Bigl( \frac{\log m n}{\log x} \exp( s(g,x) ) \Bigr) \Bigr] \nonumber \\
&= \sum_{n \le x} \frac{h_0 (n)}{n} \cdot \frac{1}{x} \sum_{k \le x} g(k) + O \Bigl( \sum_{n \le x} \frac{|\tilde{h}(n)|}{n} \sum_{m \le x/n} \frac{|h(m)|}{m} \frac{\log m + \log n}{\log x} \exp(s(g, x))  \Bigr)
\end{align}
where $h_0(n) = \tilde{h} * h(n) = \sum_{m | n} \tilde{h}(n/m) h(m)$. From \eqref{g}, we know that
\begin{equation} \label{sum-g}
s(g, x) = \sum_{x^\epsilon < p \le x} \frac{|1 - f(p)|}{p} \; \; \text{ and } \; \; s(f, x^\epsilon) + s(g, x) = \sum_{p \le x} \frac{|1 - f(p)|}{p}.
\end{equation}
Since $|h(p^k)| \le \left\{ \begin{array}{ll} |f(p) - 1|, & \text{ if } p \le x^\epsilon; \\ 0, & \text{ if  } p > x^\epsilon \end{array} \right.$, we have
\begin{equation} \label{basic1-}
\sum_{m = 1}^{\infty} \frac{|h(m)|}{m} \le \prod_{p \le x^\epsilon} \Bigl(1 + \frac{|f(p) - 1|}{p} + \frac{2}{p^2} + \frac{2}{p^3} + \cdots \Bigr) \ll \exp( s(f, x^\epsilon) ).
\end{equation}
By $|h (p^k)| \le 2$ for $p \le x^\epsilon$ and $|h(p^k)| = 0$ for $p > x^\epsilon$, we obtain
\begin{align} \label{basic2-}
\sum_{m = 1}^{\infty} \frac{|h (m)|}{m} \log m =& \sum_{m = 1}^{\infty} \frac{|h (m)|}{m} \sum_{d | m} \Lambda(d) \nonumber \\
=& \sum_{p \le x^\epsilon} \log p \mathop{\sum_{m = 1}^{\infty}}_{p | n} \frac{|h (m)|}{m} + \mathop{\sum_{p^k \le x}}_{k \ge 2} \log p \mathop{\sum_{m = 1}^{\infty}}_{p^k | m} \frac{|h (m)|}{m}\nonumber \\
\le& \sum_{p \le x^\epsilon} \log p \, \Bigl(\frac{2}{p} + \frac{2}{p^2} + \cdots \Bigr) \sum_{m = 1}^{\infty} \frac{|h(m)|}{m} + \sum_{p} \log p \, \Bigl(\frac{2}{p^2} + \frac{2}{p^3} + \cdots \Bigr) \sum_{m = 1}^{\infty} \frac{|h(m)|}{m} \nonumber \\
\ll& (\epsilon \log x + 1) \exp( s(f, x^\epsilon) )
\end{align}
by \eqref{basic1-} and Merten's estimate. Hence, the error term in \eqref{doubleconvolute} is
\begin{equation} \label{basic2}
\ll \sum_{n \le x} \frac{|\tilde{h}(n)|}{n} \Bigl( \epsilon + \frac{1}{\log x} \Bigr) \exp(s(f, x)) + \sum_{n \le x} \frac{|\tilde{h}(n)| \log n}{n \log x} \exp(s(f, x)) \ll \Bigl( \epsilon + \frac{1}{\log x} \Bigr) \exp(s(f, x))
\end{equation}
by \eqref{sum-g}, \eqref{basic1-}, \eqref{basic2-} and the fact that $\tilde{h}$ is supported over powerful numbers.

\bigskip

When $p \le x^\epsilon$, we have
\begin{align*}
h_0 (p^k) =& (f(p^k) - f(p^{k-1}) f(p))(1) + \sum_{i = 1}^{k-1} (f(p^{k-i}) - f(p^{k-i-1}) f(p)) (f(p)^i - f(p)^{i-1}) + (1) (f(p)^k - f(p)^{k-1}) \\
=& f(p^k) - f(p^{k-1}) f(p) + \sum_{i = 1}^{k-1} f(p^{k-i}) f(p)^i - \sum_{i = 1}^{k-1} f(p^{k-i-1}) f(p)^{i+1} - \sum_{i = 1}^{k-1} f(p^{k-i}) f(p)^{i-1} \\
&+ \sum_{i = 1}^{k-1} f(p^{k-i-1}) f(p)^{i} + f(p)^k - f(p)^{k-1} \\
=& f(p^k) - f(p^{k-1}).
\end{align*}
When $p > x^\epsilon$, we have
\[
h_0 (p^k) = (f(p^k) - f(p^{k-1}) f(p))(1) + 0 + \cdots + 0 = \left\{ \begin{array}{ll} 0, & \text{ if } k = 1; \\ f(p^k) - f(p^{k-1}) f(p), & \text{ if } k > 1. \end{array} \right.
\]
Thus, $|h_0 (p^k)| \le 2$ for all prime power $p^k$ and $h_0(p) = \left\{ \begin{array}{ll} f(p) - 1, & \text{ if } p \le x^\epsilon; \\ 0, & \text{ if  } p > x^\epsilon \end{array} \right.$. These imply
\begin{equation*}
\sum_{n = 1}^{\infty} \frac{|h_0(n)|}{n} \le \prod_{p \le x^\epsilon} \Bigl(1 + \frac{|f(p) - 1|}{p} + \frac{2}{p^2} + \frac{2}{p^3} + \cdots \Bigr) \prod_{p > x^\epsilon} \Bigl(1 + \frac{2}{p^2} + \frac{2}{p^3} + \cdots \Bigr) \ll \exp( s(f, x^\epsilon) ).
\end{equation*}
and  the series $\sum_{n = 1}^{\infty} \frac{h_0(n)}{n}$ converges absolutely to
\begin{align} \label{basic4}
\prod_{p \le x^\epsilon} & \Bigl(1 + \frac{f(p)}{p} + \frac{f(p^2)}{p^2} + \cdots \Bigr) \Bigl(1 - \frac{1}{p} \Bigr) \prod_{p > x^\epsilon} \Bigl(1 + \frac{f(p^2) - f(p) f(p)}{p^2} + \frac{f(p^3) - f(p^2) f(p)}{p^3} + \cdots \Bigr) \nonumber \\
&= \Theta(f, x^\epsilon) \exp \Bigl( O \big( \sum_{p > x^\epsilon} \frac{1}{p^2} + \frac{1}{p^3} +  \cdots \bigr) \Bigr) = \Theta(f, x^\epsilon) \Bigl(1 + O \bigl( x^{-\epsilon} \bigr) \Bigr)
\end{align}
Therefore,
\[
\frac{1}{x} \sum_{n \le x} f(n) = \Theta(f, x^\epsilon) \Bigl(1 + O \bigl( x^{-\epsilon} \bigr) \Bigr) \cdot \frac{1}{x} \sum_{k \le x} g(k) + O \bigl( \epsilon \exp (s(f, x)) \bigr)
\]
by \eqref{doubleconvolute}, \eqref{basic2} and \eqref{basic4}. This gives Proposition \ref{prop4.5} with a slightly adjusted main term. 
\end{proof}

Similar to Proposition \ref{prop4.5}, one has the following logarithmic version.
\begin{proposition} \label{prop8.2}
Let $f$ be any multiplicative function with $|f(n)| \le 1$. Let $g$ be the completely multiplicative function defined by $g(p) = 1$ for $p \le y$ and $g(p) = f(p)$ for $p > y$. Then, for $2 \le y \le x / 2$,
\begin{equation} \label{prop8.2eq}
\frac{1}{\log x} \sum_{n \le x} \frac{f(n)}{n} = \Theta(f, y) (1 + O(y^{-1})) \cdot \frac{1}{\log x} \sum_{n \le x} \frac{g(n)}{n} + O \Bigl( \frac{\log y}{\log x} \exp( s(f, y)) \Bigr),
\end{equation}
where $s(f, y) = \sum_{p \le y} |1 - f(p)| / p$. The remainder term above is $\ll (\log y)^3 / \log x$.
\end{proposition}

\begin{proof}
This is basically Proposition 8.2 in \cite{GS1}. Since its proof is almost identical to that of Proposition \ref{prop4.5}, we will only give a sketch. Instead of Proposition \ref{prop4.1}, one can use the following trivial estimate
\begin{equation} \label{log-prop4.1}
\frac{1}{\log x} \sum_{n \le x} \frac{f(n)}{n} - \frac{1}{\log (x / z)} \sum_{n \le x / z} \frac{f(n)}{n} \ll \frac{\log 2z}{\log x}
\end{equation}
which is valid for all functions $f$ with $|f(n)| \le 1$, and all $1 \le z \le \frac{x}{2y}$. Using \eqref{log-prop4.1}, we obtain
\begin{align} \label{log-doubleconvolute}
\frac{1}{\log x} \sum_{n \le x} \frac{f(n)}{n} =& \sum_{n \le x} \frac{\tilde{h}(n)}{n} \sum_{m \le x/n} \frac{h(m)}{m} \cdot \frac{\log x - \log (mn)}{\log x} \cdot \frac{1}{\log \frac{x}{mn}} \sum_{k \le x / (m n)} \frac{g(k)}{k} \nonumber \\
=& \frac{1}{\log x} \sum_{n \le x} \frac{\tilde{h}(n)}{n} \mathop{\sum_{m \le x/n}}_{m n \le x / (2y)} \frac{h(m)}{m} \cdot \Bigl[ \sum_{k \le x} \frac{g(k)}{k} + O \Bigl( \frac{\log m n}{\log x} \Bigr) \Bigr] \nonumber \\
&+ O \Bigl( \sum_{n \le x} \frac{|\tilde{h}(n)|}{n} \mathop{\sum_{m \le x/n}}_{m n \le x / (2y)} \frac{|h(m)|}{m} \frac{\log m + \log n}{\log x} \cdot \frac{1}{\log \frac{x}{mn}} \sum_{k \le x / (m n)} \frac{|g(k)|}{k} \Bigr) \nonumber \\ 
&+ \frac{1}{\log x} \sum_{n \le x} \frac{\tilde{h}(n)}{n} \mathop{\sum_{m \le x/n}}_{m n > x / (2y)} \frac{h(m)}{m} \sum_{k \le x / (mn)} \frac{g(k)}{k}
=: S_1 + S_2 + S_3
\end{align}
similar to \eqref{doubleconvolute}. Since $x / (mn) < 2y$ and $|g| \le 1$, one can show that $S_3$ is bounded by the error term in \eqref{prop8.2eq} via \eqref{basic1-}. Similarly, one can show that the error $S_2$ and the error term from $S_1$ are also bounded by the error term in \eqref{prop8.2eq}. The main term in $S_1$ gives the main term in \eqref{prop8.2eq}.
Note that the inequality \eqref{log-prop4.1} was stated for the range $1 \le z \le \sqrt{x}$ in \cite{GS1}. However, that is not sufficient as our proof indicates.
\end{proof}

\section{Auxiliary multiplicative functions over powerful numbers}

Suppose $f \in \hat{\mathcal{F}}_2(S)$ is a multiplicative function supported on powerful numbers. We mimic the auxiliary functions constructed in our modified proof of Proposition \ref{prop4.5}. Define a completely multiplicative function $\tilde{f}$ supported on perfect squares by
\[
\tilde{f}(p^2) := f(p^2), \; \; \tilde{f}(p^{2k}) := f(p^2)^k, \; \text{ and } \; \tilde{f}(p^{2k+1}) := 0.
\]
Also, define the multiplicative function
\begin{equation} \label{hfull}
\tilde{h}(p^k) := \left\{ \begin{array}{ll} 0, & \text{ if } k = 1; \\
f(p^k) - f(p^{k-2}) f(p^2), & \text{ if } k \ge 2. \end{array} \right.
\end{equation}
Then, one can check that
\[
f(n) = \sum_{m | n} \tilde{h} \Bigl( \frac{n}{m} \Bigr) \tilde{f}(m)
\]
as
\[
f(p^k) = \left\{ \begin{array}{ll} \sum_{0 \le i < k/2} \bigl( f(p^{k - 2i}) - f(p^{k - 2i -2}) f(p^2) \bigr) f(p^2)^i + f(p^2)^k, & \text{ if $k$ is even}; \\
\sum_{0 \le i < k/2} \bigl( f(p^{k - 2i}) - f(p^{k - 2i -2}) f(p^2) \bigr) f(p^2)^i, & \text{ if $k$ is odd}. \end{array} \right.
\]
Now, we define a completely multiplicative function over all natural numbers by
\[
\tilde{g}(m) := \tilde{f}(m^2) \; \text{ with } \; \tilde{g}(p^k) := \tilde{f}(p^{2k}) = f(p^2)^k.
\]
Then,
\begin{equation} \label{mean}
\frac{1}{\sqrt{x}} \sum_{n \le x} f(n) = \sum_{n \le x} \frac{\tilde{h}(n)}{\sqrt{n}} \cdot \frac{1}{\sqrt{x/n}} \sum_{m \le \sqrt{x/n}} \tilde{g}(m)
\end{equation}
and
\begin{equation} \label{logmean}
\sum_{n \le x} \frac{f(n)}{\sqrt{n}} = \sum_{n \le x} \frac{\tilde{h}(n)}{\sqrt{n}} \sum_{m \le \sqrt{x/n}} \frac{\tilde{g}(m)}{m}.
\end{equation}
Note that $\tilde{h}$ is supported over $3$-full (or cubefull) numbers because of \eqref{hfull}.

\section{Proof of Theorem \ref{thm-logstruct}}

\begin{proof}
We shall give a detailed proof for the structure of logarithmic spectrum here. The proof for $\hat{\Gamma}_{\mathcal{Q}_2}(S)$ follows closely and easily from the proof of Theorem 3 in \cite{GS1} by using identity \eqref{mean}. We borrow ideas from \cite{GS3} and \cite{GS1}. From identity \eqref{logmean} and $|\tilde{g}(n)| \le 1$, we have
\begin{equation} \label{log1}
\sum_{n \le x} \frac{f(n)}{\sqrt{n}} = \sum_{n \le (\log x)^{12}} \frac{\tilde{h}(n)}{\sqrt{n}} \sum_{m \le \sqrt{x/n}} \frac{\tilde{g}(m)}{m} + O \Bigl(\log x \sum_{n > (\log x)^{12}} \frac{|h(n)|}{\sqrt{n}} \Bigr).
\end{equation}
Since $\tilde{h}$ is supported over cubefull numbers and $|\tilde{h}(n)| \le d(n) \ll n^{0.01}$, the above error term is
\begin{equation} \label{log2}
\ll \log x \int_{(\log x)^{12}}^{\infty} \frac{d F_3(u)}{u^{0.49}} \ll \log x \int_{(\log x)^{12}}^{\infty} \frac{F_3(u)}{u^{1.49}} du \ll \frac{\log x}{(\log x)^{12 \times 0.15}} = \frac{1}{(\log x)^{0.8}}
\end{equation}
by partial summation and $Q_3(x) = \sum_{n \le x, n \text{ cubefull }} 1 \ll x^{1/3}$. With $\tilde{G}(x) := \frac{1}{x} \sum_{n \le x} \tilde{g}(n)$, we have
\begin{align} \label{log3}
\sum_{\sqrt{x/n} < m \le \sqrt{x}} \frac{\tilde{g}(m)}{m} &= \int_{\sqrt{x/n}}^{\sqrt{x}} \frac{d \; t \tilde{G}(t)}{t} = \tilde{G}(\sqrt{x}) - \tilde{G} \Bigl( \sqrt{\frac{x}{n}} \Bigr) + \int_{\sqrt{x/n}}^{\sqrt{x}} \frac{\tilde{G}(t)}{t} dt \nonumber \\
&= \tilde{G}(\sqrt{x}) \int_{\sqrt{x/n}}^{\sqrt{x}} \frac{1}{t} dt + O \Bigl( \frac{1}{(\log x)^{0.2}} \Bigr) = \tilde{G}(\sqrt{x}) \log \sqrt{n} + O \Bigl( \frac{1}{(\log x)^{0.2}} \Bigr)
\end{align}
by partial summation and Proposition \ref{slowdecay} as $n \le (\log x)^{12}$. Putting \eqref{log2} and \eqref{log3} into \eqref{log1}, we get
\begin{align} \label{temp1}
\sum_{n \le x} \frac{f(n)}{\sqrt{n}} &= \sum_{n \le (\log x)^{12}} \frac{\tilde{h}(n)}{\sqrt{n}} \sum_{m \le \sqrt{x}} \frac{\tilde{g}(m)}{m} - \frac{1}{\sqrt{x}} \sum_{m \le \sqrt{x}} \tilde{g}(m)  \sum_{n \le (\log x)^{12}} \frac{\tilde{h}(n) \log \sqrt{n}}{\sqrt{n}} + O \Bigl( \frac{1}{(\log x)^{0.2}} \Bigr) \nonumber \\
&= H_0 \sum_{m \le \sqrt{x}} \frac{\tilde{g}(m)}{m} + H_1 \frac{1}{\sqrt{x}} \sum_{m \le \sqrt{x}} \tilde{g}(m) + O \Bigl( \frac{1}{(\log x)^{0.2}} \Bigr)
\end{align}
with
\[
H_0 := \sum_{n = 1}^{\infty} \frac{\tilde{h}(n)}{\sqrt{n}} = \prod_{p} \Bigl(1 + \frac{f(p^3)}{p^{3/2}} + \frac{f(p^4) - f(p^2) f(p^2)}{p^{4/2}} + \frac{f(p^5) - f(p^3) f(p^2)}{p^{5/2}} + \cdots \Bigr) 
\]
and
\[
H_1 := - \frac{1}{2} \sum_{n = 1}^{\infty} \frac{\tilde{h}(n) \log n}{\sqrt{n}}
\]
by $|\tilde{h}(n) \log n| \le d(n) \log n \ll n^{0.01}$, estimate \eqref{log2}, and $\sum_{m \le x} \frac{|\tilde{g}(m)|}{m} \le \sum_{m \le x} \frac{1}{m} \ll \log x$. Applying Proposition \ref{prop8.2} with $f = \tilde{g}$ to the first term in \eqref{temp1} and estimate the second term in \eqref{temp1} trivially, we get
\begin{align} \label{temp2}
\sum_{n \le x} \frac{f(n)}{\sqrt{n}} =& \prod_{p \le y} \Bigl(1 + \frac{f(p^3)}{p^{3/2}} + \frac{f(p^4) - f(p^2) f(p^2)}{p^{4/2}} + \frac{f(p^5) - f(p^3) f(p^2)}{p^{5/2}} + \cdots \Bigr) \Bigl( \frac{1 - \frac{1}{p}}{1 - \frac{f(p^2)}{p}} \Bigr) \Bigl(1 + O \Bigl( \frac{1}{y} \Bigr) \Bigr) \nonumber \\
&\times \sum_{n \le \sqrt{x}} \frac{g(n)}{n} + O(\log^3 y) \nonumber \\
=& \prod_{p \le y} \Bigl(1 + \frac{f(p^3)}{p^{3/2}} + \frac{f(p^4) - f(p^2) f(p^2)}{p^{4/2}} + \frac{f(p^5) - f(p^3) f(p^2)}{p^{5/2}} + \cdots \Bigr) \Bigl( \frac{1 - \frac{1}{p}}{1 - \frac{f(p^2)}{p}} \Bigr) \sum_{n \le \sqrt{x}} \frac{g(n)}{n} \nonumber \\
&+ O\Bigl(\frac{\log x}{y} + \log^3 y \Bigr)
\end{align}
where $g$ is the completely multiplicative function defined by $g(p) = 1$ for $p \le y$ and $g(p) = f(p^2)$ for $p > y$. Finally, one can simply follow the proofs of Theorem 3 or Theorem 8.4 in \cite{GS1} with $y = \exp( (\log x)^{1/4} )$ and obtain the following structure result
\begin{equation} \label{temp-struct}
\hat{\Gamma}_{0, \mathcal{Q}_2} (S) = \hat{\Gamma}_{\Theta} (S) \times \Lambda_0 (S)
\end{equation}
where
\begin{align*}
\hat{\Gamma}_{\Theta} (S) := \lim_{x \rightarrow \infty} \Bigl\{ & \prod_{p \le x} \Bigl(1 + \frac{f(p^3)}{p^{3/2}} + \frac{f(p^4) - f(p^2) f(p^2)}{p^{4/2}} + \frac{f(p^5) - f(p^3) f(p^2)}{p^{5/2}} + \cdots \Bigr) \Bigl( \frac{1 - \frac{1}{p}}{1 - \frac{f(p^2)}{p}} \Bigr) \Bigl( \frac{1}{1 + \frac{1}{p^{3/2}}} \Bigr) \\
& :  f \in \hat{\mathcal{F}}_2 (S) \Bigr\} 
\end{align*}
by \eqref{log-powerfulcount}. This finishes the proof of Theorem \ref{thm-logstruct}.
\end{proof}

\section{Proof of Theorem \ref{thm-main}}

\begin{proof}
Suppose $S = [-1, 1]$. Unlike the situation in \cite{GS3}, the $H_0$ in \eqref{temp1} may be negative, coming from $p = 2$ term in the Euler product. Apply \eqref{LambdaSpec-use} to \eqref{temp-struct}, we obtain
\begin{equation} \label{temp-struct2}
\hat{\Gamma}_{0, \mathcal{Q}_2} ([-1, 1]) = \hat{\Gamma}_{\Theta} ([-1, 1]) \times [0, 1].
\end{equation}
By imitating the proof of Theorem \ref{thm4} as in \cite{GS1}, one can show that
\[
\hat{\Gamma}_{\Theta}([-1, 1]) \subset \hat{\Gamma}_{\Theta}([-1, 1]) \times \mathcal{E}([-1, 1]) \subset \hat{\Gamma}_{\Theta}([-1, 1]). 
\]
as $1 \in \mathcal{E}([-1, 1]) = [0,1]$. In particular,
\begin{equation} \label{euler-new}
\hat{\Gamma}_{\Theta}([-1, 1]) = \hat{\Gamma}_{\Theta}([-1, 1]) \times [0, 1].
\end{equation}
Thus, to understand the modified Euler product spectrum, it suffices to find the most positive and most negative values in $\hat{\Gamma}_{\Theta}([-1, 1])$ as \eqref{euler-new} implies that $\hat{\Gamma}_{\Theta}([-1, 1])$ is simply the interval between these two extreme values.

\bigskip

Now, observe that
\begin{align} \label{eulerpiece}
I_p &:=  1 + \frac{f(p^3)}{p^{3/2}} + \frac{f(p^4) - f(p^2) f(p^2)}{p^{4/2}} + \frac{f(p^5) - f(p^3) f(p^2)}{p^{5/2}} + \frac{f(p^6) - f(p^4) f(p^2)}{p^{6/2}} \cdots  \nonumber \\
&= 1 - \frac{f(p^2)^2}{p^{4/2}} + \frac{f(p^3) (p - f(p^2))}{p^{5/2}} + \frac{f(p^4) (p - f(p^2))}{p^{6/2}} + \cdots
\end{align}
When $p \ge 3$, we have $\frac{1 - 1/p}{1 - f(p^2) / p} > 0$ and
\begin{align} \label{ip}
I_p &\ge 1 - \frac{1}{p^{4/2}} - \frac{p+1}{p^{5/2}} - \frac{p+1}{p^{6/2}} - \cdots = 1 - \frac{1}{p^2} - \frac{p+1}{p^2 (p^{1/2} - 1)} \nonumber \\
&= 1 - \frac{1}{p^2} - \frac{1}{p (p^{1/2} - 1)} - \frac{1}{p^2 (p^{1/2} - 1)} \ge 1 - \frac{1}{3^2} - \frac{1}{3 (3^{1/2} - 1)} - \frac{1}{3^2 (3^{1/2} - 1)} > 0.
\end{align}
Denote $\alpha_p := f(p^2)$. Then
\begin{align} \label{bigportion}
I_p \cdot \Bigl( \frac{1 - \frac{1}{p}}{1 - \frac{f(p^2)}{p}} \Bigr) \Bigl( \frac{1}{1 + \frac{1}{p^{3/2}}} \Bigr) &\le \Bigl(1 - \frac{\alpha_p^2}{p^2} + \frac{p - \alpha_p}{p^{5/2}} + \frac{p - \alpha_p}{p^{6/2}} + \cdots \Bigr) \Bigl( \frac{p - 1}{p - \alpha_p} \Bigr) \Bigl( \frac{1}{1 + \frac{1}{p^{3/2}}} \Bigr) \nonumber \\
&= \Bigl( \frac{p - 1}{p - \alpha_p} \Bigr) \Bigl(\frac{p^2 - \alpha_p^2}{p^2} + \frac{p - \alpha_p}{p^{5/2} (1 - p^{-1/2})} \Bigr) \Bigl( \frac{1}{1 + \frac{1}{p^{3/2}}} \Bigr) \nonumber \\
& = (p-1) \Bigl( \frac{p + \alpha_p}{p^2} + \frac{1}{p^{5/2} - p^2} \Bigr) \Bigl( \frac{1}{1 + \frac{1}{p^{3/2}}} \Bigr) \nonumber \\
&\le (p - 1) \Bigl( \frac{p + 1}{p^2} + \frac{1}{p^{5/2} - p^2} \Bigr) \Bigl( \frac{1}{1 + \frac{1}{p^{3/2}}} \Bigr) = 1
\end{align}
where the upper bound can be achieved when $f(p^k) = 1$ for all $k \ge 2$. When $p = 2$, equation \eqref{eulerpiece} gives
\begin{align} \label{smallportion}
I_2 \cdot \Bigl( \frac{1 - \frac{1}{2}}{1 - \frac{\alpha_2}{2}} \Bigr) \Bigl( \frac{1}{1 + \frac{1}{2^{3/2}}} \Bigr) &\ge \frac{2 - 1}{2 - \alpha_2} \cdot \Bigl(1 - \frac{\alpha_2^2}{2^2} - \frac{2 - \alpha_2}{2^{5/2}} - \frac{2 - \alpha_2}{2^{6/2}} + \cdots \Bigr) \Bigl( \frac{1}{1 + \frac{1}{2^{3/2}}} \Bigr) \nonumber \\
&= \frac{1}{2 - \alpha_2} \cdot \Bigl( \frac{2^2 - \alpha_2^2}{2^2} - \frac{2 - \alpha_2}{2^{5/2}(1 - 2^{-1/2})} \Bigr) \Bigl( \frac{1}{1 + \frac{1}{2^{3/2}}} \Bigr) \nonumber \\
&= \Bigl( \frac{2 + \alpha_2}{2^2} - \frac{1}{2^{5/2} - 2^2} \Bigr) \Bigl( \frac{1}{1 + \frac{1}{2^{3/2}}} \Bigr) \nonumber \\
&\ge \Bigl( \frac{1}{4} - \frac{1}{2^{5/2} - 4} \Bigr) \Bigl( \frac{1}{1 + \frac{1}{2^{3/2}}} \Bigr) = - \frac{\sqrt{2}}{4 + \sqrt{2}}
\end{align}
where the lower bound can be achieved when $f(2^k) = -1$ for all $k \ge 2$. Combining \eqref{bigportion} and \eqref{smallportion}, the most negative value of $\hat{\Gamma}_{\Theta}([-1, 1])$ is $- \frac{\sqrt{2}}{4 + \sqrt{2}}$ while the most positive value is $1$. Combining this with \eqref{temp-struct2} and \eqref{euler-new}, we have the logarithmic spectrum part of Theorem \ref{thm-main}.

\bigskip

Now, for the regular spectrum, we use $\hat{\Gamma}_{\mathcal{Q}_2} ([-1, 1]) = \hat{\Gamma}_{\Theta} ([-1, 1]) \times \Lambda([-1, 1])$ from Theorem \ref{thm-logstruct}. It was proved in \cite{GS1} that $\Lambda([-1, 1]) = [ \delta_1, 1] = [-0.656999 \ldots, 1]$. This together with $\hat{\Gamma}_{\Theta}([-1, 1]) = [ \delta_2, 1] = [-0.26120 \ldots, 1]$ imply $\hat{\Gamma}_{\mathcal{Q}_2} ([-1, 1]) = [\delta_1, 1]$.

\end{proof}

\bibliographystyle{amsplain}

Mathematics Department \\
Kennesaw State University \\
Marietta, GA 30060 \\
tchan4@kennesaw.edu

\end{document}